\documentclass[12pt]{amsart}

\newtheorem{theorem}{Theorem}[section]

\newtheorem{corollary}[theorem]{Corollary}
\theoremstyle{definition}

\theoremstyle{remark}

\numberwithin{equation}{section}

\usepackage{times}
\usepackage{enumerate}
\usepackage{graphicx,adjustbox}
\usepackage{hyperref}
\usepackage{amsmath,amssymb}
\usepackage{amscd}
\usepackage{graphicx}
\usepackage[all]{xy}

\title[Betti numbers of Bresinsky's curves in $\mathbb{A}^{4}$]
{Betti numbers of Bresinsky's curves in $\mathbb{A}^{4}$}
\author{
Ranjana Mehta
\and
Joydip Saha
\and
Indranath Sengupta
}
\date{}

\address{\small \rm  Discipline of Mathematics, IIT Gandhinagar, Palaj, Gandhinagar, 
Gujarat 382355, INDIA.}
\email{ranjssj16@gmail.com}

\address{\small \rm  Discipline of Mathematics, IIT Gandhinagar, Palaj, Gandhinagar, 
Gujarat 382355, INDIA.} 
\email{saha.joydip56@gmail.com}

\address{\small \rm  Discipline of Mathematics, IIT Gandhinagar, Palaj, Gandhinagar, 
Gujarat 382355, INDIA.}
\email{indranathsg@iitgn.ac.in}
\thanks{The third author is the corresponding author.}

\date{}



\begin{document}

\begin{abstract}
Bresinsky defined a class of monomial curves in $\mathbb{A}^{4}$ with the property 
that the minimal number of generators or the first Betti number of the defining ideal 
is unbounded above. We prove that the same behaviour of unboundedness is true 
for all the Betti numbers and construct an explicit minimal free resolution for the defining 
ideal of this class of curves.
\end{abstract}

\maketitle

\section{Bresinsky's Examples}
Let $r\geq 3$ and $n_{1}, \ldots, n_{r}$ be positive integers with 
$\gcd (n_1,\ldots,\, n_r)=1$. Let us assume that the numbers 
$n_{1}, \ldots, n_{r}$ generate the numerical semigroup 
$$\Gamma(n_1,\ldots, n_r) = \lbrace\sum_{j=1}^{r}z_{j}n_{j}\mid z_{j}\quad \text{nonnegative \, integers}\rbrace$$ 
minimally, that is if $n_i=\sum_{j=1}^{r}z_{j}n_{j}$ for some non-negative 
integers $z_{j}$, then $z_{j}=0$ for all $j\neq i$ and $z_{i}=1$. Let 
$\eta:k[x_1,\,\ldots,\, x_r]\rightarrow k[t]$ be the mapping defined by 
$\eta(x_i)=t^{n_i},\,1\leq i\leq r$, where $k$ is a field. 
Let $\frak{p}(n_1,\ldots, n_r) = \ker (\eta)$. 
Let $\beta_{i}(\frak{p}(n_1,\ldots, n_r))$ denote the $i$-th Betti number of the ideal 
$\frak{p}(n_1,\ldots, n_r)$. Therefore, $\beta_{1}(\frak{p}(n_1,\ldots, n_r))$ denotes the 
minimal number of generators $\frak{p}(n_1,\ldots, n_r)$. For a given $r\geq 3$, 
let $\beta_{i}(r) = { \sup}(\beta_{i}(\frak{p}(n_1,\ldots, n_r)))$, 
where $\sup$ is taken over all the sequences of positive integers $n_1,\ldots, n_r$. 
Herzog \cite{herzog} proved that $\beta_{1}(3)$ is $3$ and it follows easily that 
$\beta_{2}(3)$ is a finite integer as well. Bresinsky in \cite{bre}, \cite{bre1}, \cite{bre2}, \cite{bre3}, \cite{brehoa}, 
extensively studied relations among the 
generators $n_1,\ldots, n_r$ of the numerical semigroup defined by these integers. It was proved 
in  \cite{bre1} and \cite{bre2} respectively that, for $r=4$ and for certain cases 
in $r=5$, the symmetry condition on the semigroup generated by $n_1,\ldots, n_r$ 
imposes an upper bound on the first Betti number $\beta_{1}(\frak{p}(n_1,\ldots, n_r))$. 
This remains an open question in general whether symmetry condition on the numerical 
semigroup generated by $n_1,\ldots, n_r$ imposes an upper bound on 
$\beta_{1}(\frak{p}(n_1,\ldots, n_r))$. Bresinsky \cite{bre} 
constructed a class of monomial curves in $\mathbb{A}^{4}$ to prove that 
$\beta_{1}(4)=\infty$. He used this observation to prove that 
$\beta_{1}(r)=\infty$, for every $r\geq 4$. Our aim in this article is to prove in 
Theorem \ref{mainthm} that for Bresinsky's examples $\beta_{i}(4) = \infty$, for every 
$1\leq i\leq 3$ and also describe all the syzygies explicitly in \ref{second} and 
\ref{secondsyzygy}. A similar study has been carried out by J. Herzog and D.I. Stamate 
in \cite{herstamate} and \cite{stamate}. However, the objective and approach in our study are quite 
different. The main theorem and underlying objective of our work can be found after the description 
of Bresinsky's examples.
\medskip

Let us recall Bresinsky's example of monomial curves in $\mathbb{A}^{4}$, as 
defined in \cite{bre}. Let $q_2\geq 4$ be even. $q_{1} = q_{2}+1,\, d_1=q_{2}-1$. Set 
$n_{1}=q_{1}q_{2},\, n_{2}=q_{1}d_{1},\, n_{3}=q_{1}q_{2}+d_{1},\, 
n_{4}=q_{2}d_{1}$. It is clear that $\gcd (n_1,\, n_2,\, n_3,\, n_4)=1$. For the 
rest of the article let us use the shorthand $\mathbf{\underline{n}}$ to denote 
Bresinsky's sequence of integers defined above. Bresinsky \cite{bre} proved 
that the set $A= A_{1}\cup A_{2}\cup \{g_1, g_2\}$ generates the ideal 
$\mathfrak{p}(n_1,\ldots, n_4)$, where $A_{1}=\{f_{\mu}| f_{\mu}=x_{1}^{\mu-1} x_{3}^{{q_2}-
\mu}-x_{2}^{{q_2}-\mu}x_{4}^{\mu+1},\quad 1\leq\mu\leq q_{2}\}$, 
$A_{2}=\{f| f=x_{1}^{\nu_{1}} x_{4}^{\nu_{4}}-x_{2}^{\mu_{2}}x_{3}^{\mu_{3}},\, \nu_{1},\ \mu_{3}< d_{1}\}$ and $g_{1}=x_{1}^{d_{1}}-{x_{2}}^{q_{2}}$, $g_{2}=x_{3}x_{4}-x_{2}x_{1}$. Let us first state the main theorems proved in this paper: 

\begin{theorem}\label{mainthm}
Let $S=A_{1}\cup A_{2}^{'}\cup \{g_1, g_2\}$, where $A_{2}'=\{h_{m}\mid x_{1}^{m}x_{4}^{(q_{1}-m)}-x_{2}^{(q_{2}-m)}x_{3}^{m}, 1\leq m\leq q_{2}-2\}$. 
\begin{enumerate}[(i)]
\item $S$ is a minimal generating set for the ideal $\mathfrak{p}(\mathbf{\underline{n}})$;
\smallskip

\item $S$ is a Gr\"{o}bner basis for $\mathfrak{p}(\mathbf{\underline{n}})$ 
with respect to the lexicographic monomial order induced by $x_{3}>x_{2}>x_{1}>x_{4}$ on 
$k[x_{1}, \ldots , x_{4}]$; 
\smallskip

\item $\beta_{1}(\mathfrak{p}(\mathbf{\underline{n}})) = \mid S \mid  = 2q_{2}$;
\smallskip

\item $\beta_{2}(\mathfrak{p}(\mathbf{\underline{n}}))= 4(q_{2}-1)$;
\smallskip

\item $\beta_{3}(\mathfrak{p}(\mathbf{\underline{n}}))= 2q_{2}-3$.

\item A minimal free resolution for the ideal $\mathfrak{p}(\mathbf{\underline{n}})$ over 
the polynomial ring $R = K[x_{1}, x_{2}, x_{3}, x_{4}]$ is 
$$0\longrightarrow R^{2q_{2}-3} \stackrel{P}{\longrightarrow} R^{4(q_{2} - 1)} \stackrel{N}{\longrightarrow} R^{2q_{2}} \longrightarrow R \longrightarrow R/\mathfrak{p}(\mathbf{\underline{n}})\longrightarrow 0,$$
where 
$$P=\left[\boldsymbol{\delta}_{1}\ldots \boldsymbol{\delta}_{q_{2}-2}\mid \boldsymbol{\xi}\mid \boldsymbol{\zeta}\mid \boldsymbol{\eta}\mid \boldsymbol{\kappa}_{1}\ldots \boldsymbol{\kappa}_{q_{2}-4} \right]_{4(q_{2}-1)\times 2q_{2}-3}$$ 
$$N=\left[\boldsymbol{\beta}_{1}\ldots \boldsymbol{\beta}_{q_{2}-1} \mid \boldsymbol{\gamma}^{'}_{1}\ldots\boldsymbol{\gamma}^{'}_{q_{2}-3}\mid \boldsymbol{\alpha}^{'}\mid \boldsymbol{\beta}^{'}_{2}\ldots \boldsymbol{\beta}^{'}_{q_{2}-2}\mid\boldsymbol{\alpha}_{1}\ldots \boldsymbol{\alpha}_{q_{2}-1}  \mid -\boldsymbol{\gamma}\mid \boldsymbol{\beta}^{'}\mid \boldsymbol{\gamma}^{'} \right]_{2q_{2}\times 4(q_{2}-1)}.$$
\end{enumerate}
\end{theorem}
\medskip

The proof of the theorem is divided into various lemmas, theorems and corollaries in sections 2, 3 and 4. 
We first prove that a special subset of binomials form a minimal generating set as well as a Gr\"{o}bner basis 
for the ideal $\mathfrak{p}(\mathbf{\underline{n}})$ with respect to a suitable monomial order; see parts (i) and (ii) 
of Theorem \ref{mainthm}. We then compute the syzygy modules using this Gr\"{o}bner basis explicitly and 
minimally in \ref{second} and \ref{secondsyzygy}. We have not only computed all the total Betti numbers but 
also have written a minimal free resolution explicitly; see parts (iii) - (vi) in Theorem \ref{mainthm} and 
\ref{minres}. However, in order to determine the minimality of the first syzygy module, we have used the 
second Betti number for these ideals calculated in \cite{stamate}. It should be mentioned here that a minimal generating set of binomials for  $\mathfrak{p}(\mathbf{\underline{n}})$ has also been calculated in 
\cite{herstamate}. The authors have also computed a minimal standard basis in \cite{herstamate} and 
that has been used to calculate the Betti numbers in \cite{stamate}. While the description of the tangent 
cone in \cite{herstamate} has been used to compute the Betti numbers in \cite{stamate}, we on the other 
hand have imitated Bresinsky's approach and studied the generators of the ideal 
$\mathfrak{p}(\mathbf{\underline{n}})$ and its syzygies, leading to a complete description of a minimal 
free resolution of the defining ideal. 
\medskip

This work grew out in an attempt to understand and generalize Bresinsky's construction 
of the numerical semigroups in arbitrary embedding dimension. What is certainly interesting 
is that $n_{1}+n_{2} = n_{3}+n_{4}$, for the sequence of integers 
$\mathbf{\underline{n}} = (n_{1}, n_{2}, n_{3}, n_{4})$ given by Bresinsky. We have 
initiated a study of numerical semigroups defined by a sequence of integers formed 
by concatenation of two arithmetic sequences and we believe that such semigroups with correct 
conditions would finally give us a good model of numerical semigroups in arbitrary embedding 
dimension with unbounded Betti numbers; see \cite{mss1}, \cite{mss3}.
\medskip

\section{The first and the second Betti numbers}
A minimal generating set for the ideal $\mathfrak{p}(\mathbf{\underline{n}})$ has 
also been constructed in \cite{herstamate}. Our construction differs from that in 
\cite{herstamate} in one binomial. The main idea is to identify the set 
$A_{2}'\subset A_{2}$, defined in the statement of \ref{mainthm}, in order 
to extract a minimal generating set out of the generating set 
constructed by Bresinsky \cite{bre}. One can show that set 
$S=A_{1}\cup\{g_{1},g_{2}\}\cup A_{2}^{'}$ is a minimal generating 
for the ideal $\mathfrak{p}(\mathbf{\underline{n}})$. Therefore, 
$\beta_{1}(\mathfrak{p}(\mathbf{\underline{n}})) = \mid S\mid =2q_{2}$.
\medskip

What is really interesting is that the set $S$, minimaly generating the ideal 
$\mathfrak{p}(\mathbf{\underline{n}})$ is also a Gr\"{o}bner basis. Therefore, the Schreyer 
tuples generate the syzygy module. The following theorem proves these facts and 
finally we calculate the second Betti number by extracting a minimal generating 
set for the syzygy module.
\medskip

\begin{theorem}\label{gbasis}
Consider the lexicographic monomial order induced by 
$x_{3}> x_{2}>x_{1}>x_{4}$ in $k[x_{1},x_{2},x_{3},x_{4}]$. Then, 
\begin{enumerate}[(i)]
\item The set $S$ forms Gr\"{o}bner basis for the ideal $\mathfrak{p}(\mathbf{\underline{n}})$ with respect to the above monomial order.

\item Let $\mathbb{D}$ denote the set of all Schreyer tuples obtained from 
the Gr\"{o}bner basis $S$, which generate the first syzygy module 
(see Theorem 1.43 \cite{dl}). Then each entry in the elements of 
$\mathbb{D}$ is either a non-constant polynomial or zero.
\end{enumerate}
\end{theorem}

\begin{proof}
We first order the set $S=A_{1}\cup\{g_{1},g_{2}\}\cup A_{2}^{'}$ as follows:
$$(f_{1},\ldots,f_{q_{2}},g_{1},g_{2},h_{1},\ldots,h_{q_{2}-2}).$$ 
Let $f,g\in S$. We consider the $S$-polynomials $S(f,g)$ and divide the proof into cases 
based on the sets $f$ and $g$ belonging to. 
\smallskip

\noindent\textbf{Case 1.} $\mathbf{f,g\in A_{1}}$.
\smallskip

\noindent\textbf{1(a).}  $f=f_{\mu},\, g=f_{\mu+1}$ where, $1\leq \mu \leq q_{2}-1$. We have 
\begin{align*}
S(f_{\mu},f_{\mu+1}) & = x_{1}\cdot f_{\mu}-x_{3}\cdot f_{\mu+1}\\
{} & = (x_{2}^{q_{2}-1-\mu}x_{4}^{\mu+1})\cdot g_{2}\longrightarrow_{S} 0.
\end{align*}
\noindent Therefore, the set
$$\mathbb{T}_{1}=\{\boldsymbol{\beta}_{\mu}=(\beta_{(\mu, 1)},\cdots, \beta_{(\mu, 2q_{2})})\mid 1\leq \mu\leq q_{2}-1\}$$ 
\noindent gives the Schreyer tuples, where
\begin{eqnarray*}
\beta_{(\mu, \mu)} & = & -x_{1};\\
\beta_{(\mu, \mu+1)} & = & x_{3};\\
\beta_{(\mu, q_{2}+2)} & = & x_{2}^{q_{2}-(\mu+1)}x_{4}^{\mu+1};\\
\beta_{(\mu, i)} & = & 0,\, \text{for}\,  i\notin\{ \mu,\mu+1,q_{2}+2\}.
\end{eqnarray*}
\smallskip

\noindent \textbf{1(b).} $f=f_{\mu},\, g=f_{\mu '} $ where, $ \mu ' > \mu+1$. We have \begin{align*}
S(f_{\mu},f_{\mu '})&={x_{1}}^{\mu ' -\mu}\cdot f_{\mu}-x_{3}^{\mu '-\mu}\cdot f_{\mu '}\\
&=x_{2}^{q_{2}-\mu '}x_{4}^{\mu+1}((x_{3}x_{4})^{\mu '-\mu-1}+(x_{3}x_{4})^{\mu '-\mu-2}(x_{1}x_{2})+\cdots +(x_{1}x_{2})^{\mu '-\mu-1})\cdot g_{2}\\
{} & \longrightarrow_{S} 0 
\end{align*}
\noindent Therefore, the set 
$$\mathbb{T}_{2}=\{\boldsymbol{\gamma}_{\mu \mu '}=(\gamma_{(\mu\mu ', 1)},\ldots , \gamma_{(\mu\mu ', i)}, \ldots , \gamma_{(\mu\mu ', 2q_{2})})| 1\leq \mu\leq q_{2}-2, \, \mu+2 \leq \mu ' \leq q_{2}\}$$
\noindent gives the Schreyer tuples, where 
\begin{eqnarray*}
\gamma_{(\mu\mu ',\mu )}& = & -{x_{1}}^{(\mu '- \mu)};\\
\gamma_{(\mu\mu ',\mu ' )} & = & {x_{3}}^{(\mu '- \mu)};\\
\gamma_{(\mu\mu ',q_{2}+2)}& = & x_{2}^{q_{2}-\mu '}x_{4}^{\mu+1}((x_{3}x_{4})^{\mu '-\mu-1}+(x_{3}x_{4})^{\mu '-\mu-2}(x_{1}x_{2})+\cdots +(x_{1}x_{2})^{\mu '-\mu-1});\\
\gamma_{(\mu\mu ',i)}& = & 0 \quad \text{for}\quad  i\notin\{\mu, \mu ', (q_{2}+2)\}.
\end{eqnarray*}
\smallskip

\noindent \textbf{Case 2.}  $\mathbf{f\in A_{1},\, g\in\{g_{1},g_{2}\}}$
\smallskip

\noindent \textbf{2(a).} Let $f=f_{\mu},\, g=g_{1}$, where $1\leq \mu \leq q_{2}$.
\noindent We have
\begin{align*}
S(f_{\mu},\, g_{1})&=x_{2}^{q_{2}}\cdot f_{\mu}+x_{1}^{\mu-1}x_{3}^{q_{2}-\mu}\cdot g_{1}\\
&=x_{1}^{q_{2}-1}\cdot f_{\mu}+x_{2}^{q_{2}-\mu}x_{4}^{\mu+1}\cdot g_{1}\quad \longrightarrow_{S} 0
\end{align*}
\noindent Therefore, the set 
$$\mathbb{T}_{3}=\{\boldsymbol{\gamma}_{\mu}: \gamma_{\mu}=(\gamma_{(\mu, 1)},\ldots,\gamma_{(\mu,i)},\ldots, \gamma_{(\mu,2q_{2})})\mid 1\leq \mu\leq q_{2}-1\}$$
gives us the Schreyer tuples, where
\begin{eqnarray*}
\gamma_{(\mu,\mu)}& = & x_{1}^{q_{2}-1}-x_{2}^{q_{2}};\\
\gamma_{(\mu, q_{2}+1)} & = & x_{2}^{q_{2}-\mu}x_{4}^{\mu+1}-x_{1}^{\mu-1}x_{3}^{q_{2}-\mu};\\
\gamma_(\mu,i) & = & 0,\quad \text{for}\quad  i\notin \{ \mu, (q_{2}+1)\},\, 1\leq i\leq 2q_{2}.
\end{eqnarray*}
\smallskip

\noindent \textbf{2(b).} Let $f=f_{\mu},\, g=g_{2}$, where $1\leq \mu \leq q_{2}-1$.
We have 
\begin{align*}
S(f_{\mu},\, g_{2})&= x_{4}\cdot f_{\mu}-x_{1}^{\mu-1}x_{3}^{q_{2}-(\mu+1)}\cdot g_{2}\\
&=x_{2}\cdot f_{\mu+1} \quad \longrightarrow_{S} 0
\end{align*}
\noindent Therefore, the set
$$\mathbb{T}_{4}=\{\boldsymbol{\alpha}_{\mu}=(\alpha_{(\mu, 1)},\ldots,\alpha_{(\mu,i)},\ldots, \alpha_{(\mu,2q_{2})})\mid 1\leq \mu \leq q_{2}-1\}$$
\noindent gives the Schreyer tuples, where
 \begin{eqnarray*}
\alpha_{(\mu,\mu)}& = & -x_{4};\\
\alpha_{(\mu,\mu +1)} & = & x_{2};\\
\alpha_{(\mu, q_{2}+2)} & = & x_{1}^{\mu-1}x_{3}^{q_{2}-(\mu+1)};\\
\alpha_{(\mu,i)} & = & 0,\quad \text{for}\quad  i\notin \{\mu,\, \mu+1,\, (q_{2}+2)\},\, 1\leq i\leq 2q_{2}. 
\end{eqnarray*}
\smallskip

\noindent\textbf{2(c).} Let $f=f_{q_{2}},\, g=g_{2}$. We have
\begin{align*}
S(f_{q_{2}}, g_{2})&=x_{3}x_{4}\cdot f_{q_{2}}-x_{1}^{q_{2}-1}\cdot g_{2}\\
&=x_{1}x_{2}\cdot f_{q_{2}}-{x_{4}^{q_{2}+1}}\cdot g_{2} \quad \longrightarrow_{S} 0
\end{align*}
\noindent Therefore, the set 
$$\mathbb{T}_{5}=\{\boldsymbol{\alpha}=(\alpha_{1},\ldots,\alpha_{i},\ldots, \alpha _{ 2q_{2}})\}$$ 
\noindent gives the Schreyer tuples, where
\begin{eqnarray*}
\alpha_{q_{2}} & = & x_{1}x_{2}-x_{3}x_{4};\\
\alpha_{(q_{2}+2)} & = & x_{1}^{q_{2}-1}-x_{4}^{q_{2}+1};\\
\alpha_{i} & = & 0,\quad \text{for}\quad  i\notin \{q_{2},\,(q_{2}+2)\},\, 1\leq i\leq 2q_{2}.
\end{eqnarray*}
\smallskip

\noindent\textbf{Case 3.} Let $\mathbf{f=g_{1}, g=g_{2}}$. We have
\begin{align*}
S(g_{1},g_{2})&=-x_{3}x_{4}\cdot g_{1}-x_{2}^{q_{2}}\cdot g_{2}\\
&=-x_{1}x_{2}\cdot g_{1}-x_{1}^{q_{2}-1}\cdot g_{2} \quad \longrightarrow_{S} 0
\end{align*}
\noindent Therefore, the set 
$$\mathbb{T}_{6}=\{\boldsymbol{\beta}=(\beta_{1},\ldots,\beta_{i},\ldots, \beta_{ 2q_{2}})\}$$
gives the Schreyer tuples, where 
\begin{eqnarray*}
\beta_{q_{2}+1} & = & x_{3}x_{4}-x_{1}x_{2};\\ 
\beta_{q_{2}+2} & = & x_{2}^{q_{2}}-x_{1}^{q_{2}-1};\\
\beta_{i} & = & 0,\quad \text{for}\quad  i\notin \{(q_{2}+1),\, (q_{2}+2)\},\, 1\leq i\leq 2q_{2}.
\end{eqnarray*}
\smallskip

\noindent\textbf{Case 4.} $\mathbf{f\in \{g_{1},g_{2}\}, g\in A_{2}^{'}}$. 
\smallskip

\noindent\textbf{4(a).} Let $f=g_{1},\, g=h_{1}$. We have
\begin{align*}
S(g_{1},\, h_{1})&=-x_{3}\cdot g_{1}+x_{2}\cdot h_{1}\\
&=-x_{1}\cdot f_{q_{2}-1}\quad \longrightarrow_{S} 0
\end{align*}
\noindent Therefore, the set
$$\mathbb{T}_{7}=
\{\boldsymbol{\gamma}=(\gamma_{1},\ldots,\gamma_{i},\ldots, \gamma_{2q_{2}})\}$$ 
gives us the Schreyer tuples, where
\begin{eqnarray*}
\gamma_{q_{2}-1} & = & -x_{1};\\
\gamma_{q_{2}+1} & = & x_{3};\\
\gamma_{q_{2}+3} & = & -x_{2};\\ 
\gamma_{i} & = & 0,\quad \text{for}\quad  i\notin \{(q_{2}-1),\, (q_{2}+1),\,(q_{2}+3)\}.
\end{eqnarray*}
\smallskip

\noindent\textbf{4(b).} Let $f=g_{1}, g=h_{m}$, with $1< m\leq (q_{2}-2)$. We have
\begin{align*}
S(g_{1}, h_{m})&=-{x_{3}}^{m}\cdot g_{1}+{x_{2}}^{m}\cdot h_{m}\\
&=-{x_{1}}^{m}\cdot f_{q_{2}-m}\quad \longrightarrow_{S} 0
\end{align*}
\noindent Therefore, the set
$$\mathbb{T}_{8}=\{\boldsymbol{\alpha}'_{m}=(\alpha'_{(m,1)},\ldots,\alpha'_{(m,i)},\ldots, \alpha'_{(m,2q_{2})}|1< m\leq q_{2}-2\}$$
gives us the Schreyer tuples, where
\begin{eqnarray*}
\alpha'_{(m,q_{2}-m)} & = & -x_{1}^{m};\\
\alpha'_{(m,q_{2}+2)} & = & x_{3}^{m};\\ 
\alpha'_{(m,q_{2}+2+m)} & = & -x_{2}^{m};\\ 
\alpha'_{(m,i)} & = & 0,\quad \text{for}\quad  i\notin \{q_{2}-m, q_{2}+1, q_{2}+2+m\}.
\end{eqnarray*}
\smallskip

\noindent\textbf{4(c).} Let $f=g_{2}, g=h_{1}$. We have 
\begin{align*}
S(g_{2},h_{1})&=x_{2}^{q_{2}-1}\cdot g_{2}+x_{4}\cdot h_{1}\\
&=x_{1}\cdot g_{1}-x_{1}\cdot f_{q_{2}} \quad \longrightarrow_{S} 0
\end{align*}
\noindent Therefore, the set 
$$\mathbb{T}_{9}=\{\boldsymbol{\alpha}'=(\alpha'_{1},\ldots, \alpha'_{i},\ldots , \alpha'_{2q_{2}})\}$$
gives us the Schreyer tuples, where
\begin{eqnarray*} 
\alpha'_{q_{2}} & = & -x_{1};\\
\alpha'_{(q_{2}+1)} & = & x_{1};\\ 
\alpha'_{(q_{2}+2)} & = & -x_{2}^{(q_{2}-1)};\\ 
\alpha'_{(q_{2}+3)} & = & -x_{4};\\ 
\alpha'_{i} & = & 0, \quad \text{for}\quad  i\notin \{ q_{2},\,q_{2}+1,\,q_{2}+2,\,q_{2}+3\}.
\end{eqnarray*}
\medskip

\noindent\textbf{4(d).} Let $f=g_{2}, g=h_{m},$ where $1< m \leq q_{2}-2.$ We have
\begin{align*}
S(g_{2},h_{m})&=x_{2}^{q_{2}-m}x_{3}^{m-1}\cdot g_{2}-x_{4}\cdot h_{m}\\
&=-x_{1}\cdot h_{m-1} \quad \longrightarrow_{S} 0
\end{align*}
\noindent Therefore, the set
$$\mathbb{T}_{10}= \{\boldsymbol{\beta}'_{m}=(\beta'_{(m,1)},\ldots,\beta'_{(m,i)},\ldots, \beta'_{(m,2q_{2})})|1< m\leq q_{2}-2)\}$$
gives us the Schreyer tuples, where
\begin{eqnarray*}
\beta'_{(m,q_{2}+2)} & = & x_{2}^{q_{2}-m}x_{3}^{m-1};\\
\beta'_{(m,q_{2}+m+1)} & = & -x_{1};\\
\beta'_{(m,q_{2}+2+m)} & = & x_{4};\\
\beta'_{(m,i)} & = & 0,\quad \text{for}\quad  i\notin \{(q_{2}+2), (q_{2}+m+1), (q_{2}+2+m)\}.
\end{eqnarray*}
\medskip

\noindent\textbf{Case 5.} $\mathbf{f, g\in A_{2}^{'}}$
\smallskip

\noindent\textbf{5(a)} Let $f=h_{m}$ and $g=h_{m+1},$ where $1\leq m \leq q_{2}-3$. We have 
\begin{align*}
S(h_{m},h_{m+1})&= -x_{3}\cdot h_{m}+x_{2}\cdot h_{m+1}\\
&=x_{1}^{m}x_{4}^{q_{2}-m}\cdot g_{2}\quad \longrightarrow_{S} 0
\end{align*}
\noindent Therefore the set
$$\mathbb{T}_{11}=\{{\boldsymbol{\gamma}'_{m}}=(\gamma'_{(m, 1)},\ldots,\gamma'_{(m,i)},\ldots, \gamma'_{(m, 2q_{2})})\mid 1\leq m\leq q_{2}-3\}$$
gives us the Schreyer tuples, where
\begin{eqnarray*}
\gamma'_{(m, q_{2}+2)} & = & -x_{1}^{m}x_{4}^{q_{2}-m};\\
\gamma'_{(m,q_{2}+2+m)} & = & x_{3};\\
\gamma'_{(m ,q_{2}+3+m)} & = & -x_{2};\\
\gamma'_{(m, i)} & = & 0,\quad \text{for}\quad i\notin \{q_{2}+2, q_{2}+m+2,q_{2}+m+3\}.
\end{eqnarray*}
\medskip

\noindent\textbf{5(b).} Let $f=h_{m}$ and $g=h_{m'}$ where $m'> m+1.$ We have
\begin{align*}
S(h_{m}, h_{m'})&=-x_{2}^{m'-m}\cdot h_{m}+x_{3}^{m'-m}\cdot h_{m'}\\
&=x_{1}^{m}x_{4}^{q_{2}+1-m'}((x_{3}x_{4})^{m'-m-1}+(x_{3}x_{4})^{m'-m-2}(x_{1}x_{2})+\cdots+(x_{1}x_{2})^{m-m'-1})\cdot g_{2}\\
{} & \longrightarrow_{S} 0 
\end{align*}
\noindent Therefore, the set
\begin{align*}
\mathbb{T}_{12}=\{{\boldsymbol{\alpha}'_{m m'}}=(\alpha'_{(m m', 1)},\ldots,\alpha'_{(m m',i)},\ldots, \alpha'_{(m m', 2q_{2})})\mid & 1\leq m\leq q_{2}-4,\\
& m+2\leq m'\leq q_{2}-2\}
\end{align*}
gives the Schreyer tuples, where
\begin{eqnarray*}
\alpha'_{(mm',m )} & = & x_{3}^{m'-m};\\
\alpha'_{(mm', m')} & = & -x_{2}^{m'-m};\\ 
\alpha'_{(mm',q_{2}+2)} & = & x_{1}^{m}x_{4}^{q_{2}+1-m'}((x_{3}x_{4})^{m'-m-1}+(x_{3}x_{4})^{m'-m-2}(x_{1}x_{2})+\cdots+(x_{1}x_{2})^{m-m'-1});\\
\alpha'_{(mm', i)} & = & 0,\quad \text{for}\quad  i\notin \{ m,\, m',\, q_{2}+2\}.
\end{eqnarray*}
\medskip

\noindent \textbf{Case 6.} $\mathbf{f\in A}$ and $\mathbf{g\in A_{2}^{'}}$
\smallskip

\noindent\textbf{6(a).} Let $f=f_{1},\, g=h_{q_{2}-2}.$ We have
\begin{align*}
S(f_{1}, h_{q_{2}-2})&= x_{2}^{2}\cdot f_{1}+x_{3}\cdot h_{q_{2}-2}\\
&=x_{1}^{q_{2}-2}x_{4}^{2}\cdot g_{2}+x_{2}x_{4}^{2}\cdot g_{1}\quad \longrightarrow_{S} 0
\end{align*}
\noindent Therefore, the set
$$\mathbb{T}_{13}=\{\boldsymbol{\beta}'=(\beta'_{1},\ldots,\beta'_{i},\ldots, \beta'_{ 2q_{2}})\}$$
gives the Schreyer tuples, where
\begin{eqnarray*}
\beta'_{1} & = & -x_{2}^{2};\\
\beta'_{(q_{2}+1)} & = & x_{2}x_{4}^{2};\\
\beta'_{(q_{2}+2)} & = & x_{1}^{q_{2}-2}x_{4}^{2};\\
\beta'_{2q_{2}} & = & -x_{3};\\
\beta'_{i} & = & 0,\quad \text{for}\quad i\notin \{ 1,\, (q_{2}+1),\, (q_{2}+2),\,2q_{2}\},\, 1\leq i\leq 2q_{2}.
\end{eqnarray*}
\medskip

\noindent\textbf{6(b).} Let $f=f_{2},\, g=h_{q_{2}-2}.$ We have
\begin{align*}
S(f_{2}, h_{q_{2}-2})&= x_{2}^{2}\cdot f_{2}+x_{1}\cdot h_{q_{2}-2}\\
&=x_{4}^{3}\cdot g_{1}\quad \longrightarrow_{S} 0
\end{align*}
Therefore, the set 
$$\mathbb{T}_{14}=\{\boldsymbol{\gamma}'=(\gamma'_{1},\ldots,\gamma'_{i},\ldots, \gamma'_{ 2q_{2}})\}$$
gives the Schreyer tuples, where
\begin{eqnarray*}
\gamma'_{2} & = & -x_{2}^{2};\\
\gamma'_{(q_{2}+1)} & = & x_{4}^{3};\\
\gamma'_{(2q_{2})} & = & -x_{1};\\
\gamma'_{i} & = & 0,\quad \text{for}\quad  i\notin \{ 2,\,(q_{2}+1),\, (2q_{2})\},\, 1\leq i\leq 2q_{2}.
\end{eqnarray*}
\medskip

\noindent\textbf{6(c).} Let $f=f_{\mu}$ and $g=h_{m},\, (\mu,m)\neq(1,\, q_{2}-2)$ and 
$(\mu,m)\neq(2,\, q_{2}-2)$.
\medskip

\noindent\textbf{(i)} $\mu + m < q_{2}.$ We have
\begin{align*}
S(f_{\mu}, h_{m})&= x_{2}^{q_{2}-m}\cdot f_{\mu}+x_{1}^{\mu-1}x_{3}^{q_{2}-\mu-m}\cdot h_{m}\\
&=x_{4}^{q_{2}+1-m}\cdot f_{\mu+m}-x_{2}^{q_{2}-\mu-m}x_{4}^{\mu+1}\cdot (f_{q_{2}}-g_{1})\quad \longrightarrow_{S} 0
\end{align*}
\noindent Therefore, the set
\begin{align*}\mathbb{T}_{15}=\{{\boldsymbol{\beta}'_{\mu m}}=(\beta'_{(\mu m, 1)},\ldots,\beta'_{(\mu m,i)},\ldots, \beta'_{(\mu m, 2q_{2})})\mid & 1\leq \mu \leq q_{2}-1,\\
& 1\leq m< q_{2}-\mu \}
 \end{align*}
gives the Schreyer tuples, where
\begin{eqnarray*}
\beta'_{(\mu m, \mu)} & = & -x_{2}^{q_{2}-m};\\
\beta'_{(\mu m, \mu +m)} & = & x_{4}^{q_{2}+1-m};\\
\beta'_{(\mu m ,q_{2})} & = & -x_{2}^{q_{2}-m-\mu}x_{4}^{\mu+1};\\
\beta'_{(\mu m ,q_{2}+1)} & = & x_{2}^{q_{2}-m-\mu}x_{4}^{\mu+1};\\
\beta'_{(\mu m ,q_{2}+2+m)} & = & -x_{1}^{\mu-1}x_{3}^{q_{2}-\mu-m};\\
\beta'_{(m, i)} & = & 0,\quad \text{for}\quad  i\notin 
\{ \mu, \, \mu+m ,\,  q_{2},\, (q_{2}+1), \, (q_{2}+2+m)\}.
\end{eqnarray*}
\medskip
 
\noindent\textbf{(ii)} Let $ \mu + m=q_{2}.$ We have
\begin{align*}
S(f_{\mu}, h_{m})&= x_{2}^{q_{2}-m}\cdot f_{\mu}+x_{1}^{\mu-1}\cdot h_{m}\\
&=x_{4}^{\mu+1}\cdot g_{1}\quad \longrightarrow_{S} 0
\end{align*}
\noindent Therefore, the set
\begin{align*}
\mathbb{T}_{16}=\{{\boldsymbol{\gamma}'_{\mu m}}=(\gamma'_{(\mu m, 1)},\ldots,\gamma'_{(\mu m,i)},\ldots, \gamma'_{(\mu m, 2q_{2})})\mid  1\leq & \mu \leq q_{2}-1,\\
&  m = q_{2}-\mu \}
\end{align*}
gives us the Schreyer tuples, where
\begin{eqnarray*} 
\gamma'_{(\mu m, \mu)} & = & -x_{2}^{q_{2}-m};\\
\gamma'_{(\mu m ,q_{2}+1)}& = & x_{4}^{\mu+1};\\
\gamma'_{(\mu m ,2q_{2}+2-\mu)}& = & -x_{1}^{\mu-1};\\
\gamma'_{(\mu m, i)} & = & 0,\quad \text{for}\quad  
i\notin \{ \mu,\, (q_{2}+1), \, (2q_{2}+2-\mu)\}.
\end{eqnarray*}
\medskip

\noindent\textbf{(iii)} Let $\mu+m > q_{2}.$ We have
 \begin{align*}
S(f_{\mu}, h_{m})&= x_{2}^{q_{2}-m}x_{3}^{\mu+m-q_{2}}\cdot f_{\mu}+x_{1}^{\mu-1}\cdot h_{m}\\
&=x_{4}^{\mu +1}\cdot h_{\mu+m-q_{2}}+x_{1}^{\mu+m-q_{2}}x_{4}^{q_{2}+1-m}\cdot f_{q_{2}}\quad \longrightarrow_{S} 0
\end{align*}
\noindent Therefore, the set
\begin{align*}
\mathbb{T}_{17}=\{{\boldsymbol{\alpha}''_{\mu m}}=(\alpha''_{(\mu m, 1)},\ldots,\alpha''_{(\mu m,i)},\ldots, \alpha''_{(\mu m, 2q_{2})})\mid & 1\leq \mu \leq q_{2}-1,\\
& 1\leq m< q_{2}-\mu \}
\end{align*}
gives us the Schreyer tuples, where
\begin{eqnarray*} 
\alpha''_{(\mu m, \mu)} & = & -x_{2}^{q_{2}-m}x_{3}^{\mu+m-q_{2}};\\
\alpha''_{(\mu m ,q_{2})} & = & x_{1}^{\mu+m-q{2}}x_{4}^{q_{2}+1-m};\\
\alpha''_{(\mu m ,\mu+m+2)} & = & x_{4}^{\mu+1};\\
\alpha''_{(\mu m ,q_{2}+2+m)} & = & -x_{1}^{\mu-1};\\
\alpha''_{(m, i)} & = & 0,\quad \text{for}\quad  i\notin \{ \mu ,\, q_{2},\, (q_{2}+2+m), \, (\mu+m+2)\}.
\end{eqnarray*}

\noindent The set $\mathbb{D}=\cup_{i=1}^{17}\mathbb{T}_{i}$ gives us all the 
Schreyer tuples which form the generating set for the first syzygy module 
$\mathrm{Syz}(\mathfrak{p}(\mathbf{\underline{n}}))$. \end{proof}
\medskip

\begin{theorem}\label{second} 
Let $\mathbb{T}= \mathbb{T}_{1}\cup \mathbb{T}_{4}\cup\mathbb{T}_{7}\cup\mathbb{T}_{9}\cup \mathbb{T}_{10}\cup\mathbb{T}_{11}\cup \mathbb{T}_{13}\cup\mathbb{T}_{14}$. Let 
$M_{1}$ denote the first syzygy module $\mathrm{Syz}(\mathfrak{p}(\mathbf{\underline{n}}))$. 
Then, the set $\mathbb{\overline{T}}\subset M_{1}/\mathfrak{m}M_{1}$ is linearly independent 
over the field $R/\mathfrak{m}=k$, where $\mathfrak{m} = \langle x_{1}, x_{2}, x_{3}, x_{4}\rangle$. The set $\mathbb{T}$ is a minimal generating set for the first syzygy module and 
$\beta_{2}(\mathfrak{p}(\mathbf{\underline{n}}))= 4(q_{2}-1)$.

\end{theorem}

\begin{proof}
The module $M_{1}$ is generated by $\mathbb{D}$, therefore $\mathfrak{m}M_{1}$ 
must be generated by 
$L=x_{1}\mathbb{D}\cup x_{2}\mathbb{D}\cup x_{3}\mathbb{D}\cup x_{4}\mathbb{D}$. 
It follows from the construction of $\mathbb{D}$ that each coordinate of elements 
of $L$ is either zero or a monomial of total degree greater than one. Now, to 
show that $\overline{\mathbb{T}}$ is linearly independent in the vector space 
$M_{1}/\mathfrak{m}M_{1}$ over the field $R/\mathfrak{m}=k$, we consider the 
element $\mathbf{v}\in \mathfrak{m}M_{1}$ given by
\begin{align*}
\mathbf{v} & = \sum_{\mu=1}^{q_{2}-1}a_{\mu}\boldsymbol{\beta}_{\mu}+\sum_{m=1}^{q_{2}-3}b_{m}\boldsymbol{\gamma}'_{m}+c_{1}\boldsymbol{\alpha'}+
\sum_{m=2}^{q_{2}-2}c_{m}\boldsymbol{\beta}'_{m}\\
{} & \quad\quad +\sum_{\mu=1}^{q_{2}-1}d_{\mu}\boldsymbol{\alpha_{\mu}}-l_{1}\boldsymbol{\gamma}+l_{2}\boldsymbol{\beta}'+l_{3}\boldsymbol{\gamma}' 
\end{align*}
where $a_{\mu}, b_{m},c_{m},d_{\mu},l_{m}\in k$. The 
linear part in the first coordinate $v_{1}$ is $a_{1}(-x_{1})+d_{1}(-x_{4})$. 
Each coordinate of elements of $L$ being a monomial of total degree 
greater than one or zero, we have $a_{1}(-x_{1})+d_{1}(-x_{4})=0$, that is 
$a_{1}=d_{1}=0$. Now, the linear part in the $i$-th coordinate 
$v_{i}$, for $1<i< q_{2}-1$ is $a_{i-1}(x_{3})+a_{i}(-x_{1})+d_{i-1}(x_{2})+d_{i}(-x_{4})$. 
Therefore by the same argument we must have $a_{i}=d_{i}=0$, for $1\leq i< q_{2}-1$. 
Similarly, the linear part in $v_{q_{2}+1}$ is $c_{1}(x_{1})+l_{1}(-x_{3})$. 
Therefore, for similar reasons we must have $c_{1}=l_{1}=0$. 
We now compute linear part in $v_{q_{2}-1}$ and obtain 
$a_{q_{2}-1}(-x_{1})+d_{q_{2}-1}(-x_{4})$. Equating this 
to zero we obtain $a_{q_{2}-1}=d_{q_{2}-1}=0$. It turns out that, 
$\mathbf{v}=\sum_{m=1}^{q_{2}-3}b_{m}\boldsymbol{\gamma}'_{m}+\sum_{m=2}^{q_{2}-2}c_{m}\boldsymbol{\beta}'_{m}+l_{2}\boldsymbol{\beta'}+l_{3}\boldsymbol{\gamma'}$. 
\medskip

We now compute the linear part in $v_{q_{2}+3}$ and obtain $b_{1}x_{3}+c_{2}(-x_{1})$. 
Equating this to zero we get $b_{1}=c_{2}=0$. The linear part in $v_{q_{2}+2+i}$, 
for $1<i\leq q_{2}-3$, is $\, b_{i-1}(-x_{2})+b_{i}x_{3}+c_{i}x_{4}+c_{i+1}(-x_{1})$. 
Therefore, $b_{i}=c_{i}=0$, \, for $1<i<q_{2}-2$. Finally computing the linear part 
in $v_{2q_{2}}$ and equating that to zero we obtain 
$$b_{q_{2}-3}(-x_{2})+c_{q_{2}-2}x_{4}+l_{2}(-x_{3})+l_{3}(-x_{1})=0.$$ 
Therefore, $b_{q_{2}-3}=c_{q_{2}-2}=l_{2}=l_{3}=0$. This proves that 
the set $\overline{\mathbb{T}}\subset M_{1}/\mathfrak{m}M_{1}$ is linearly independent. 
Therefore, $\mathbb{T}$ is a part of a minimal generating set 
for the first syzygies module, and cardinality of $\mathbb{T}$ is 
$$(q_{2}-1)+(q_{2}-3)+(q_{2}-2)+(q_{2}-1)+3=4(q_{2}-1).$$ 
By Theorem 8.1 \cite{stamate}, the second betti number in the resolution of 
$R/\mathfrak{p}(\mathbf{\underline{n}})$ is  $4(q_{2}-1)$, hence  $\mathbb{T}$ is a minimal generating set for first syzygy module. Therefore, 
$\beta_{2}(\mathfrak{p}(\mathbf{\underline{n}}))= \mid \mathbb{T} \mid = 4(q_{2}-1)$.
\end{proof}

\begin{theorem}\label{third}
$\beta_{3}(\mathfrak{p}(\mathbf{\underline{n}}))= 2q_{2}-3$.
\end{theorem}

\begin{proof}
We know that $R/\mathfrak{p}(\mathbf{\underline{n}})\cong k[t^{n_{1}},t^{n_{2}},t^{n_{3}},t^{n_{4}}]$ is a one dimensional 
integral domain and therefore $\mathrm{depth}(R/\mathfrak{p}(\mathbf{\underline{n}}))=1$. 
By the Auslander Buchsbaum theorem, 
$\mathrm{projdim}_{R}(R/\mathfrak{p}(\mathbf{\underline{n}}))=3$ and we 
have 
$$1-2q_{2}+\beta_{2}(\mathfrak{p}(\mathbf{\underline{n}}))-\beta_{3}(\mathfrak{p}(\mathbf{\underline{n}}))=0.$$ 
Therefore, by \ref{second}, we have 
$\beta_{3}(\mathfrak{p}(\mathbf{\underline{n}}))= 2q_{2}-3$.\end{proof}
\medskip

\section{The second syzygy and a minimal free resolution}
Let us order the generating vectors of second syzygy, and consider the matrix $$N=\left[\boldsymbol{\beta}_{1}\ldots \boldsymbol{\beta}_{q_{2}-1} \mid \boldsymbol{\gamma}^{'}_{1}\ldots\boldsymbol{\gamma}^{'}_{q_{2}-3}\mid \boldsymbol{\alpha}^{'}\mid \boldsymbol{\beta}^{'}_{2}\ldots \boldsymbol{\beta}^{'}_{q_{2}-2}\mid\boldsymbol{\alpha}_{1}\ldots \boldsymbol{\alpha}_{q_{2}-1}  \mid -\boldsymbol{\gamma}\mid \boldsymbol{\beta}^{'}\mid \boldsymbol{\gamma}^{'} \right]_{2q_{2}\times 4(q_{2}-1)}$$

\noindent We consider the following sets of vectors: 
\begin{enumerate}
\item [(i)] $\mathbb{H}_{1}=\{\boldsymbol{\delta}_{\mu}=(\delta_{(\mu,1)},\ldots \delta_{(\mu,4(q_{2}-1))})\mid 1\leq \mu\leq q_{2}-2\} $, where
\begin{eqnarray*} 
\delta_{(\mu,\mu)} & = & x_{4};\\
\delta_{(\mu,\mu+1)} & = &-x_{2};\\ 
\delta_{(\mu,(3q_{2}-6+\mu))} & = & -x_{1};\\ 
\delta_{(\mu,(3q_{2}-5+\mu))} & = & x_{3};\\ 
\delta_{(\mu,i)} & = & 0, \quad \text{for}\quad  i\notin \{\mu,\mu+1,(3q_{2}-6+\mu),(3q_{2}-5+\mu)\}.
\end{eqnarray*}
\item [(ii)] $\mathbb{H}_{2}=\{\boldsymbol{\xi}=(\xi_{1},\ldots,\xi_{4(q_{2}-1)})\}$, where
\begin{eqnarray*} 
\xi_{1} & = & x_{2}^{2};\\
\xi_{2q_{2}-3} & = & x_{2}x_{4}^{2};\\ 
\xi_{4q_{2}-7} & = & x_{1}x_{4}^{2};\\ 
\xi_{4q_{2}-6} & = & x_{4}^{3};\\
\xi_{4q_{2}-5} & = & -x_{1};\\ 
\xi_{4q_{2}-4} & = & x_{3};\\
\xi_{i} & = & 0, \quad \text{for}\quad  i\notin \{1,(2q_{2}-3),(4q_{2}-7),(4q_{2}-6),(4q_{2}-5),(4q_{2}-4)\}.
\end{eqnarray*}
\item [(iii)] $\mathbb{H}_{3}=\{\boldsymbol{\zeta}=(\zeta_{1},\ldots,\zeta_{4(q_{2}-1)})\}$, where
\begin{eqnarray*} 
\zeta_{q_{2}-1} & = & x_{1};\\
\zeta_{q_{2}} & = & x_{4};\\ 
\zeta_{2q_{2}-3} & = & x_{3};\\ 
\zeta_{2q_{2}-2} & = & x_{2};\\
\zeta_{4q_{2}-6} & = & x_{1};\\ 
\zeta_{i} & = & 0, \quad \text{for}\quad  i\notin \{(q_{2}-1),q_{2},(2q_{2}-3),(2q_{2}-2),(4q_{2}-6)\}.
\end{eqnarray*}
\item [(iv)] $\mathbb{H}_{3}=\{\boldsymbol{\eta}=(\eta_{1},\ldots,\eta_{4(q_{2}-1)})\}$, where
\begin{eqnarray*} 
\eta_{2q_{2}-4} & = & -x_{1};\\
\eta_{3q_{2}-6} & = & -x_{3};\\ 
\eta_{3q_{2}-5} & = & x_{2}^{2};\\ 
\eta_{4q_{2}-5} & = & -x_{4};\\
\eta_{4q_{2}-4} & = & x_{2};\\ 
\eta_{i} & = & 0, \quad \text{for}\quad  i\notin \{(2q_{2}-4),(3q_{2}-6),(3q_{2}-5),(4q_{2}-5),(4q_{2}-4)\}.
\end{eqnarray*}
\item [(v)] $\mathbb{H}_{4}=\{\boldsymbol{\kappa}_{\mu}=(\kappa_{(\mu,1)},\ldots,\kappa_{(\mu,4(q_{2}-1))})\mid 1\leq \mu\leq q_{2}-4\}$, where
\begin{eqnarray*} 
\eta_{(\mu,(q_{2}-1+\mu))} & = & -x_{1};\\
\eta_{(\mu,q_{2}+\mu)} & = & x_{4};\\ 
\eta_{(\mu,(2q_{2}-3+\mu))} & = & -x_{3};\\ 
\eta_{(\mu,(2q_{2}-2+\mu))} & = & x_{2};\\
\eta_{(\mu,i)} & = & 0, \quad \text{for}\quad  i\notin \{(q_{2}-1+\mu),(q_{2}+\mu),(2q_{2}-3+\mu),(2q_{2}-2+\mu)\}.
\end{eqnarray*}

\end{enumerate}
\medskip

\begin{theorem}\label{secondsyzygy}
The set $\mathbb{H}=\mathbb{H}_{1}\cup\mathbb{H}_{2}\cup\mathbb{H}_{3}\cup\mathbb{H}_{4}$ 
is a minimal generating set of the second syzygy module of 
$\mathfrak{p}(\mathbf{\underline{n}})$.
\end{theorem}

\proof Let us define the matrix,
$$P=\left[\boldsymbol{\delta}_{1}\ldots \boldsymbol{\delta}_{q_{2}-2}\mid \boldsymbol{\xi}\mid \boldsymbol{\zeta}\mid \boldsymbol{\eta}\mid \boldsymbol{\kappa}_{1}\ldots \boldsymbol{\kappa}_{q_{2}-4} \right]_{4(q_{2}-1)\times (2q_{2}-3)}.$$

The matrix $$N=\left[\boldsymbol{\beta}_{1}\ldots \boldsymbol{\beta}_{q_{2}-1} \mid \boldsymbol{\gamma}^{'}_{1}\ldots\boldsymbol{\gamma}^{'}_{q_{2}-3}\mid \boldsymbol{\alpha}^{'}\mid \boldsymbol{\beta}^{'}_{2}\ldots \boldsymbol{\beta}^{'}_{q_{2}-2}\mid\boldsymbol{\alpha}_{1}\ldots \boldsymbol{\alpha}_{q_{2}-1}  \mid -\boldsymbol{\gamma}\mid \boldsymbol{\beta}^{'}\mid \boldsymbol{\gamma}^{'} \right]_{2q_{2}\times 4(q_{2}-1)}$$ is the one which has been defined at the beginning of this section. It is easy to check that $N\cdot P=0$. Therefore, elements of $\mathbb{H}$ are elements of the second syzygy module. Let $ M_{2}$ denote the second syzygy 
module of $\mathfrak{p}(\mathbf{\underline{n}})$. We claim that, 
$\overline{\mathbb{H}}\subset M_{2}/\mathfrak{m}M_{2}$ is a linearly independent 
set, where $\mathfrak{m}=\langle x_{1},x_{2},x_{3},x_{4}\rangle$. We proceed 
in the same way as in \ref{second}, considering the expression 
$$ \sum_{\mu=1}^{q_{2}-2}p_{\mu}\boldsymbol{\delta}_{\mu}+q\boldsymbol{\xi}+w \boldsymbol{\zeta}+s\boldsymbol{\eta}+\sum_{j=1}^{q_{2}-4}t_{j}\boldsymbol{\kappa}_{j}=\mathbf{u}\,(\mathrm{say}),$$ where $p_{\mu},q,w,s,t_{j}\in k$ for $1\leq \mu\leq q_{2}-2$,  $ 1\leq j\leq q_{2}-4$ and we compute linear terms in each coordinate of $\mathbf{u}$. If we compute the linear terms 
in $u_{1}$ we get $p_{1}x_{4}$, hence $p_{1}=0$. Next we compute the linear terms in 
$u_{2}$ and we get $-p_{1}x_{2}+p_{2}x_{4}$. We have $-p_{1}x_{2}+p_{2}x_{4}=0$, hence 
$p_{2}=0$. Proceeding like this, we observe that all the coefficients are zero and 
therefore $\mathbb{H}$ is a part of minimal generating set of $M_{2}$. Since the third 
betti number is $2q_{2}-3$ and $\mid\mathbb{H}\mid =2q_{2}-3$, we conclude that 
$\mathbb{H}$ is a minimal generating set of the second syzygy module. \qed
\medskip

\begin{corollary}\label{minres}
A minimal free resolution for the ideal $\mathfrak{p}(\mathbf{\underline{n}})$ in 
the polynomial ring $R = K[x_{1}, x_{2}, x_{3}, x_{4}]$ is 
$$0\longrightarrow R^{2q_{2}-3} \stackrel{P}{\longrightarrow} R^{4(q_{2} - 1)} \stackrel{N}{\longrightarrow} R^{2q_{2}} \longrightarrow R \longrightarrow R/\mathfrak{p}(\mathbf{\underline{n}})\longrightarrow 0,$$
where 
$$P=\left[\boldsymbol{\delta}_{1}\ldots \boldsymbol{\delta}_{q_{2}-2}\mid \boldsymbol{\xi}\mid \boldsymbol{\zeta}\mid \boldsymbol{\eta}\mid \boldsymbol{\kappa}_{1}\ldots \boldsymbol{\kappa}_{q_{2}-4} \right]_{4(q_{2}-1)\times (2q_{2}-3)}$$ 
$$N=\left[\boldsymbol{\beta}_{1}\ldots \boldsymbol{\beta}_{q_{2}-1} \mid \boldsymbol{\gamma}^{'}_{1}\ldots\boldsymbol{\gamma}^{'}_{q_{2}-3}\mid \boldsymbol{\alpha}^{'}\mid \boldsymbol{\beta}^{'}_{2}\ldots \boldsymbol{\beta}^{'}_{q_{2}-2}\mid\boldsymbol{\alpha}_{1}\ldots \boldsymbol{\alpha}_{q_{2}-1}  \mid -\boldsymbol{\gamma}\mid \boldsymbol{\beta}^{'}\mid \boldsymbol{\gamma}^{'} \right]_{2q_{2}\times 4(q_{2}-1)}.$$
\end{corollary} 
\bigskip

\section*{Acknowledgement}
The third author thanks IIT Gandhinagar for the funding through 
the research project IP/IITGN/MATH/IS/201415-13. The second author 
thanks SERB, Government of India for the post-doctoral fellowship, 
under the research project EMR/2015/000776.
\bigskip

\bibliographystyle{amsalpha}

\end{document}